\documentclass[11pt]{article}
\usepackage{amsmath}
\usepackage{amssymb}
\usepackage{amsthm}
\usepackage{graphicx}
\usepackage{microtype}
\usepackage{hyperref}
\usepackage{stackrel}

\newtheorem{theorem}{Theorem}[section]
\newtheorem{lemma}[theorem]{Lemma}

\newtheorem{corollary}[theorem]{Corollary}
\theoremstyle{definition}

\newtheorem{definition}[theorem]{Definition}

\def\d{\displaystyle}

\def\x{\xi}
\def\p{\partial}

\def\P{\mathcal{P}}

\def\G{\mathcal{G}}
\def\H{\mathcal{H}}

\begin{document}
\title {\textbf{Intersection numbers in the curve complex via subsurface projections}}
\author{Yohsuke Watanabe\thanks{The author was partially supported from U.S. National Science Foundation grants DMS 1107452, 1107263, 1107367 "RNMS: Geometric Structures and Representation Varieties" (the GEAR Network).}}
\date{}
\maketitle

\begin{abstract}
A classical inequality which is due to Lickorish and Hempel says that the distance between two curves in the curve complex can be measured by their intersection number. In this paper, we show a converse version; the intersection number of two curves can be measured by the sum of all subsurface projection distances between them. As an application of this result, we obtain a coarse decreasing property of the intersection numbers of the multicurves contained in tight multigeodesics. Furthermore, by using this property, we give an algorithm for determining the distance between two curves in the curve complex. 
Indeed, such algorithms have been also found by Birman--Margalit--Menasco, Leasure, Shackleton, and Webb: we will briefly compare our algorithm with some of their algorithms, for detailed quantitative comparison of all known algorithms including our algorithm, we refer the reader to the paper of Birman--Margalit--Menasco \cite{BMM}.
\end{abstract}



\section{Introduction}	

Let $S=S_{g,n}$ be a compact surface of genus $g$ and $n$ boundary components. Throughout the paper, we assume 
\begin{itemize}
\item curves are simple, closed, essential, and not parallel to any boundary component. Arcs are simple and essential. 
\item the isotopy of curves is free and the isotopy of arcs is relative to the boundaries setwise unless we say pointwise. 
\end{itemize}
We denote the complexity of $S_{g,n}$ by $\x(S_{g,n}) =3g+n-3$ and the Euler characteristic of $S_{g,n}$ by $\chi(S_{g,n})=2-2g-n$.

In \cite{HAR}, Harvey associated the set of curves in $S$ with the simplicial complex, the \emph{curve complex} $C(S)$. Suppose $\x(S)\geq 1$. The vertices of $C(S)$ are isotopy classes of curves and the simplices of $C(S)$ are collections of curves that can be mutually realized to intersect the minimal possible geometric intersection number, which is $0$ for $\x(S)>1$, $1$ for $S=S_{1,1}$, and $2$ for $S=S_{0,4}$. 
We also review the \emph{arc complex} $A(S)$, and the \emph{arc and curve complex} $AC(S)$. Suppose $\x(S)\geq 0$, the vertices of $A(S)$ ($AC(S)$) are isotopy classes of arcs (arcs and curves) and the simplices of $A(S)$ ($AC(S)$) are collections of arcs (arcs and curves) that can be mutually realized to be disjoint in $S$.

In this paper, we focus on the $1$--skeletons of the above complexes. We assign length $1$ to each edge, then they are all geodesic metric spaces, see \cite{FM}.
Suppose $x,y\in C(S)$ and $A,B\subseteq C(S)$; we let $d_{S}(x,y)$ denote the length of a geodesic between $x$ and $y$, and we define $\d d_{S}(A,B)$ as the diameter of $A \cup B$ in $C(S).$
Suppose $x,y\in AC(S)$ and $A,B\subseteq AC(S)$; the intersection number, $i(x,y)$ is the number of minimal possible geometric intersections of $x$ and $y$ in their isotopy classes. We say $x$ and $y$ are in \emph{minimal position} if they realize the intersection number. Lastly, we define $\d i(A,B)=\sum_{a\in A,b\in B} i (a,b).$

We review the following classical inequality, which is due to Lickorish \cite{LIC} and Hempel \cite{HEM}. For the rest of this paper, we always assume the base of $\log$ functions is $2$, and we treat $\log0 =0$. 

\begin{lemma}[\cite{HEM},\cite{LIC}] \label{basic}
Let $x,y\in C(S)$. Then, $$d_{S}(x,y)\leq 2\cdot \log i(x,y)+2.$$
\end{lemma}

In this paper, we prove a converse to Lemma \ref{basic}. However, since we can easily find $x,y\in C(S)$ such that $i(x,y)\gg 0$ and $d_{S}(x,y)=2$, we need to manipulate the left--hand side of the above inequality; we take all subsurface projection distances into account.
Before we state our results, we recall a beautiful formula derived by Choi--Rafi \cite{CR}. First, we define the following.

\begin{definition}
Suppose $n,m\in \mathbb{Z}$.
\begin{itemize}
\item By $n\stackrel{K,C}{\leq} m$, we mean that there exist $K\geq1$ and $C\geq0$ so that $ n\leq Km+C.$ By $n\stackrel{K,C}{=} m$, we mean that there exist $K\geq1$ and $C\geq0$ so that $\d \frac{1}{K}\cdot m-C\leq n\leq Km+C.$ We call $K$ a \textit{multiplicative constant} and $C$ an \textit{additive constant}. 
We also use notations $n\asymp m$, $n\prec m$ instead of $n\stackrel{K,C}{=} m$, $n\stackrel{K,C}{\leq} m$ respectively. 

\item We let $[n]_{m}=0$ if $n\leq m$ and $[n]_{m}=n$ if $n> m$. We call $m$ a \textit{cut--off constant}.
\end{itemize}
\end{definition}

Recall that a \emph{marking} is a collection of curves which fill a surface.

Choi--Rafi showed

\begin{theorem}[\cite{CR}]\label{yazawa}
There exists $k$ such that for any markings $\sigma$ and $\tau$ on $S$, $$\log i(x,y)  \asymp \sum_{Z\subseteq S}[ d_{Z}(\sigma,\tau)]_{k}+\sum_{A\subseteq S} \log [d_{A}(\sigma,\tau)]_{k},$$  where the sum is taken over all subsurfaces $Z$ which are not annuli and $A$ which are annuli.
\end{theorem}

Choi--Rafi used Masur--Minsky's well--known \emph{distance formula} \cite{MM2} to derive one direction of the above quasi--equality, which is Theorem \ref{E} in the setting of markings; the argument goes back to Rafi's work in \cite{RAF}. 
In this paper, we also prove that direction in the setting of curves where we have more freedom on cut--off constants. We note that our result (in the setting of curves) and Choi-Rafi's result (in the setting of markings) are closely related, see \cite{W4}. However, in this paper, we give a more direct approach; in particular, we do not use the distance formula in our proofs. As a consequence of this approach, we can compute cut--off, additive and multiplicative constants, see Theorem \ref{E''} and Theorem \ref{E}; this effectivizes Choi--Rafi's work since those constants are not given in \cite{CR}.

Now, we state our results.  Throughout the paper, we often use the expression on the right--hand side of the formula in Theorem \ref{yazawa}. We always assume the subsurfaces with $\log$ are annuli and the subsurfaces without $\log$ are non--annuli. We also note that in some cases, we will take a sum over specific subsurfaces; therefore, we always specify them under the $ \sum$--symbols. For instance, in Theorem \ref{yazawa}, $\d \sum_{Z\subseteq S}$ and $\d \sum_{A\subseteq S}$ indicate that all subsurfaces in $S$ are considered.

Let $M=200$. We first show

\begin{theorem}\label{E''}
Suppose $\x(S)=1$. Let $x,y\in C(S)$. We have 
\begin{itemize}
\item $\d \log i(x,y)\leq k^{3} \cdot \bigg( \sum_{Z\subseteq S}[ d_{Z}(x,y)]_{k}+\sum_{A\subseteq S} \log [d_{A}(x,y)]_{k} \bigg)+k^{3}$ for all $k\geq M$.

\item $\d  \log i(x,y)\geq  \frac{1}{2}\cdot \bigg( \d\sum_{Z\subseteq S}[ d_{Z}(x,y)]_{k}+\sum_{A\subseteq S} \log [d_{A}(x,y)]_{k}  \bigg) -2 $ for all $k\geq 3M$.
\end{itemize}
\end{theorem}

In order to prove Theorem \ref{E''}, we make new observations in $\S 3$ such as Theorem \ref{4}. 

Furthermore, by an inductive argument on the complexity, we prove

\begin{theorem}\label{E}
Suppose $\x(S)>1$. Let $x,y\in C(S)$. For all $k> 0$, we have $$\log i(x,y)\leq V(k) \cdot \bigg( \sum_{Z\subseteq S}[ d_{Z}(x,y)]_{k}+\sum_{A\subseteq S} \log [d_{A}(x,y)]_{k}\bigg) + V(k)$$ where $V(k)=\big( M^{2}|\chi(S)| \cdot (k+\x(S)\cdot M)  \big)^{\x(S)+2}.$ 
\end{theorem}

As an application of the above theorems, we study the intersection numbers of the multicurves which are contained in tight multigeodesics; we will observe a special behavior of tight geodesics under subsurface projections, which is Lemma \ref{sss}, and apply the above theorems. This is a new approach to study the intersection numbers on tight multigeodesics: we obtain the most effective result known so far. For the rest of this paper, given $x,y \in C(S)$, $g_{x,y}$ will denote a geodesic between $x$ and $y$, unless we specify that it is a multigeodesic.

We show

\begin{theorem}\label{ANA}
Suppose $\x(S)\geq1$. Let $x,y\in C(S)$ and $g_{x,y} = \{x_{i}\}$ be a tight multigeodesic such that $d_{S}(x,x_{i}) = i$ for all $i$. We have 
$$ i(x_{i},y) \leq R^{i}\cdot i(x,y) \text{ for all }i,\text{ where }R=\x(S)\cdot  2^{V(M)}.$$
\end{theorem}

We note that we have a stronger statement in Theorem \ref{ANA} when $\x(S)=1$, see Lemma \ref{2}.

Indeed, we can use Theorem \ref{ANA} to compute the distance between two curves in the curve complex:

\begin{corollary}
There exists an algorithm (based on Theorem \ref{ANA}) which determines the distance between two curves in the curve complex.
\end{corollary}

We remark that such algorithms have been also found by Birman--Margalit--Menasco \cite{BMM}, Leasure \cite{LEA}, Shackleton \cite{SHA}, and Webb \cite{WEB2}. Here, we omit the details of the comparison on all known algorithms since it has been discussed carefully in the paper of Birman--Margalit--Menasco \cite{BMM}. However, for example, 
\begin{itemize}
\item our algorithm is more effective than Shackleton's algorithm \cite{SHA} since Theorem \ref{ANA} is more effective than Theorem \ref{tabi}, which was proved by Shackleton in \cite{SHA}.
\item our algorithm is more effective than Birman--Margalit--Menasco's algorithm \cite{BMM} when the distance between two curves is ``large'', while their algorithm is more effective than our algorithm when the distance between two curves is ``small''. 
\end{itemize}

\renewcommand{\abstractname}{\textbf{Acknowledgements}}
\begin{abstract}
The author thanks Kenneth Bromberg for useful discussions and continuous feedback throughout this project. The author also thanks Mladen Bestvina for suggesting to prove Theorem \ref{E''} and Theorem \ref{E} and Richard Webb for insightful discussions. Finally, the author thanks Tarik Aougab and Samuel Taylor for useful conversations.

Much of this paper was written while the author was visiting Brown University under the supervision of Jeffrey Brock, the author thanks the hospitality of his and the institute. 
\end{abstract}

\section{Background}
The goal of this section is to establish our basic tools which are subsurface projections and tight geodesics from \cite{MM2}. 

\subsection{Subsurface projections}
We let $R(A)$ be a regular neighborhood of $A$ in $S$ where $A$ is a subset of $S$. Also, we let $\P(C(S))$ and $\P(AC(S))$ be the set of subsets in each complex. 

\subsubsection{Non--annular projections}
Suppose $Z\subseteq S$ is not an annulus. 
We define the map $$i_{Z}:AC(S)\rightarrow \P(AC(Z))$$ such that $i_{Z}(x)$ is the set of arcs or a curve obtained by $x\cap Z$ where $x$ and $\partial Z$ are in minimal position.  

Also we define the map $$p_{Z}:AC(Z)\rightarrow \P(C(Z))$$ such that $p_{Z}(x)=\partial R(x\cup z\cup z')$, where $ z,z' \subseteq \partial Z$ such that $z\cap \partial(x)\neq \emptyset $ and $z'\cap \partial(x)\neq \emptyset $. See Fig. \ref{luuu}. If $x\in C(Z)$ then $p_{Z}(x)=x.$ We observe $|\{p_{Z}(x)\}|\leq 2$. 

The subsurface projection to $Z$ is the map $$\pi_{Z}=p_{Z}\circ i_{Z}:AC(S)\rightarrow \P(C(Z)).$$ If $C\subseteq AC(S)$, we define $\d \pi_{Z}(C)=\bigcup_{c\in C}\pi_{Z}(c).$

\begin{figure}[h]
 \begin{center}
  \includegraphics[width=12cm]{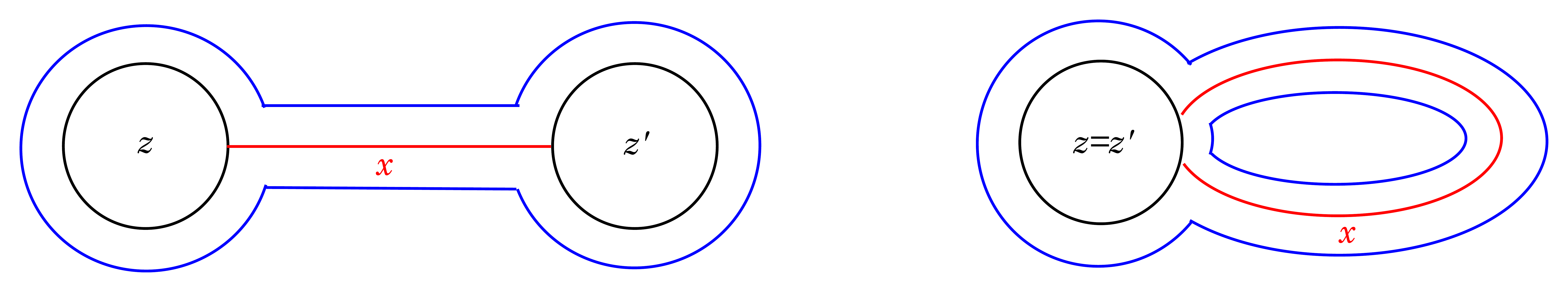}
   \end{center}
    \caption{$p_{Z}(x)=\partial R(x\cup z\cup z')$.}
   \label{luuu}
\end{figure}

Now, we observe the following. 
 
\begin{lemma}\label{kp} 
Suppose $Z\subseteq S$ is not an annulus. If $x\in AC(S)$, then $$|\{i_{Z}(x)\}|\leq 3 |\chi(S)| \text{ and } |\{\pi_{Z}(x)\}|\leq 6|\chi(S)|.$$ 

\end{lemma} 
\begin{proof}
The first inequality follows from the fact that the dimension of $AC(Z)$ is bounded by $3 |\chi(Z)|-1\leq 3 |\chi(S)|-1,$ which can be proved by an Euler characteristic argument; for instance, see \cite{KP}. 

For the second inequality, we observe that if $y\in AC(Z)$ then $|\{p_{Z}(y)\}|\leq 2.$
\end{proof}

\subsubsection{Annular projections}
Suppose $Z\subseteq S$ is an annulus. Fix a hyperbolic metric on S, compactify the cover of $S$ which corresponds to $\pi_{1}(Z)$ with its Gromov boundary, and denote the resulting surface by $S^{Z}$. Here, we define the curve complex of annuli by altering the original definition; we define the vertices of $C(Z)$ to be the isotopy classes of arcs whose endpoints lie on two boundaries of $S^{Z}$, here the isotopy is relative to $\partial S^{Z} $ \emph{pointwise}. Two vertices of $C(Z)$ are distance one apart if they can be isotoped to be disjoint in the interior of $S^{Z}$, again the isotopy is relative to $\partial S^{Z} $ \emph{pointwise}.

The subsurface projection to $Z$ is the map $$\pi_{Z}:AC(S)\rightarrow \P(C(Z))$$ such that $\pi_{Z}(x)$ is the set of all arcs obtained by the lift of $x$. As in the previous case, if $C\subseteq AC(S)$, we define $\d \pi_{Z}(C)=\bigcup_{c\in C}\pi_{Z}(c).$

\subsubsection{Subsurface projection distances}
Suppose $A,B \subseteq AC(S)$; we define $\d d_{Z}(A,B)$ as the diameter of $\pi_{Z}(A) \cup \pi_{Z}(B)$ in $C(Z).$

We recall important results regarding subsurface projection distances. 

\begin{lemma}[\cite{MM2}]\label{oct}
Suppose $\x(S)\geq1$. If $x,y\in C(S)$ such that $d_{S}(x,y)=1$, then $d_{Z}(x,y)\leq 3$ for all $Z\subseteq S$.
\end{lemma}

Now, we observe the following lemma for annular projections.

\begin{lemma}[\cite{MM2}]\label{min}
Suppose $Z$ is an annulus in $S$ and the core curve of $Z$ is $x \in C(S)$. Let $T_{x}$ be the Dehn twist along $x$. If $y\in C(S)$ such that $\pi_{Z}(y)\neq \emptyset$, then $$d_{Z}(y,T_{x}^{n}(y))=|n|+2 \text{ for }n\neq0.$$

If $y$ intersects $x$ exactly twice with opposite orientation, then the half twist about $x$ of $y$ produces a curve $H_{x}(y)$, defined by taking $x\cup y$ and resolving the intersections in a way consistent with the orientation: $H_{x}^{2}(y)=T_{x}(y)$, and $$\displaystyle d_{Z}(y,H_{x}^{n}(y))=\bigg\lfloor \frac{|n|}{2} \bigg\rfloor +2\text{ for }n\neq0.$$
\end{lemma}

Lastly, we recall the Bounded Geodesic Image Theorem which was first proved by Masur--Minsky \cite{MM2} and recently by Webb \cite{WEB1} by a more direct approach.  

\begin{theorem}[Bounded Geodesic Image Theorem]\label{BGIT}
Suppose $\x(S)\geq 1$. There exists $M$ such that the following holds. If $\{x_{i}\}_{0}^{n}$ is a (multi)geodesic in $C(S)$ and $\pi_{Z}(x_{i})\neq \emptyset$ for all $i$ where $Z\subsetneq S$, then $$d_{Z} (x_{0},x_{n}) \leq M.$$
\end{theorem}

In the rest of this paper, we mean $M$ as $M$ in the statement of the Bounded Geodesic Image Theorem. We remark that $M$ is computable, and also uniform for all surfaces. In particular, we take $M=200$ in this paper. See \cite{WEB1}.

\subsection{Tight geodesics}

A \emph{multicurve} is a set of curves that form a simplex in the curve complex. 
Let $V$ and $W$ be multicurves in $S$; we say $V$ and $W$ \emph{fill} $S$ if there is no curve in the complementary components of $V \cup W$ in $S$. Take $R(V \cup W)$ and fill in a disk for every complementary component of $R(V \cup W)$ in $S$ which is a disk, then $V$ and $W$ fill this subsurface. We denote this subsurface by $F(V,W)$. We observe
 
\begin{lemma}\label{distance3}
Suppose $\x(S)>1$. Let $V$ and $W$ be multicurves in $S$. Then, $V$ and $W$ fill $S$ if and only if $ d_{S}(V,W)>2$.
\end{lemma}

Now, we define tight geodesics.
\begin{definition}\label{d}
Suppose $\x(S)>1.$
\begin{itemize}
\item A multigeodesic is a sequence of multicurves $\{V_{i}\}$ such that $d_{S}(a,b)=|p-q|$ for all $a\in V_{p},b\in V_{q}$, and $p\neq q$. 

\item A tight multigeodesic is a multigeodesic $\{V_{i}\}$ such that $V_{i}=\partial F(V_{i-1},V_{i+1})$ for all $i$. 

\item Let $x,y \in C(S)$. A tight geodesic between $x$ and $y$ is a geodesic $\{x_{i}\}$ such that $x_{i}\in V_{i}$ for all $i$, where $\{V_{i}\}$ is a tight multigeodesic between $x$ and $y$.

\end{itemize}
\end{definition}

Masur--Minsky showed 

\begin{theorem}[\cite{MM2}]
There exists a tight geodesic between any two points in $C(S)$.
\end{theorem}

Lastly, we observe the following lemma, which states a special behavior of tight geodesics under subsurface projections: 
\begin{lemma}[\cite{W2}]\label{sss}
Suppose $\x(S)\geq 1$ and $Z\subsetneq S$. Let $x,y\in C(S)$ and $\{V_{j}\}$ be a tight (multi)geodesic between $x$ and $y$ such that $d_{S}(x,V_{j})=j$ for all $j$. If $\pi_{Z}(V_{i})\neq \emptyset$, then $$d_{Z}(x,V_{i})\leq M \text{ or } d_{Z}(V_{i},y)\leq M.$$

\end{lemma}
\begin{proof}
We assume $\x(S)>1$ since this case requires more work. Also, we assume $\{V_{j}\}$ is a tight multigeodesic. The proof applies to the case when $\{V_{j}\}$ is a tight geodesic. 

Suppose $\pi_{Z}(V_{j}) \neq \emptyset$ for all $j>i$. By Theorem \ref{BGIT}, we have $d_{Z}(V_{i},y)\leq M.$

Suppose $\pi_{Z}(V_{k})= \emptyset$ for some $k>i$. We have two cases.
\newline

\underline{If $k>i+1$}: By Lemma \ref{distance3}, we observe $\pi_{Z}(V_{j})\neq \emptyset$ for all $j< i$ since $d_{S}(V_{k},V_{j})>2$, i.e., $V_{k}$ and $V_{j}$ fill $S$. By Theorem \ref{BGIT}, we have $d_{Z}(x,V_{i})\leq M.$
\newline

\underline{If $k=i+1$}: By tightness, $V_{i}=\partial F(V_{i-1},V_{i+1})$; therefore, $Z$ must essentially intersect with $F(V_{i-1},V_{i+1})$ so that $\pi_{Z}(V_{i})\neq \emptyset$. Furthermore, we observe that $V_{i-1}$ and $V_{i+1}$ fill $F(V_{i-1},V_{i+1})$; therefore, if $\pi_{Z}(V_{i+1})= \emptyset$ then $\pi_{Z}(V_{i-1})\neq \emptyset.$ As in the previous case, we have $d_{Z}(x,V_{i})\leq M $ by Lemma \ref{distance3} and Theorem \ref{BGIT}. 

\end{proof}

\section{A Farey graph via the Bounded Geodesic Image Theorem}
Recall that the $1$--skeletons of the curve complexes of $S_{1,1}$ and $S_{0,4}$ are both Farey graphs; the vertices are identified with $\mathbb{Q}\cup \{\frac{1}{0}=\infty\}\subset S^{1}$. 

The following observation is elementary, yet useful in this section.

\begin{lemma}
Suppose $\x(S)=1$. Let $x,y\in C(S)$ such that $d_{S}(x,y)=1$. If $I$ and $I'$ are the two (open) intervals obtained by removing $\{x,y\}$ from $S^{1}$, then any geodesic between a curve in $I$ and a curve in $I'$ needs to contain $x$ or $y$.  
\end{lemma}
\begin{proof}
Since $d_{S}(x,y)=1$, there exists the edge between $x$ and $y$. The statement follows from the fact that the interiors of any two distinct edges of a Farey graph are disjoint. 
\end{proof}

We will use the above lemma throughout this section. The goal of this section is to observe Theorem \ref{4}. First, we prove Lemma \ref{2} and Lemma \ref{ru}.

\begin{lemma}\label{2}
Suppose $\x(S)=1$. Let $x,y\in C(S)$ and $g_{x,y}=\{x_{i}\}$ such that $d_{S}(x,x_{i})=i$ for all $i$. Then, for all $0<i<d_{S}(x,y)$, we have $$ \frac{i(x_{i-1},y)}{i(x_{i},y)}>\frac{3}{2}.$$
\end{lemma} 

\begin{proof}
We recall that if $\frac{s}{t},\frac{p}{q}\in C(S)$, then $i \big(\frac{s}{t},\frac{p}{q} \big)=k\cdot |sq-pt|$ where $k=1,2$ if $S=S_{1,1},S_{0,4}$ respectively.

We may assume $x_{i-1}=\frac{0}{1}$ and $x_{i}=\frac{1}{0}$. Let $y=\frac{p}{q}$; we have 
\begin{itemize}
\item $i (x_{i-1},y ) =i \big(\frac{0}{1},\frac{p}{q} \big)= k\cdot|p|.$
\item $i (x_{i},y ) =i \big(\frac{1}{0},\frac{p}{q} \big)= k\cdot|q|.$
\end{itemize}
Thus, we have $$|y|= \frac{|p|}{|q|}=  \frac{k\cdot |p|}{k\cdot |q|}= \frac{i(x_{i-1},y)}{i(x_{i},y)}.$$ Therefore, it suffices to prove $|y|>\frac{3}{2}.$

\emph{Proof.}
We suppose $|y|\leq \frac{3}{2}$, and derive a contradiction. Assume $y>0$, the same argument works for the case $y<0$. 
\newline

\underline{Assume $y\leq 1$}:
Since there exists the edge between $x_{i-1}=0$ and $1$, $g_{x_{i},y}$ needs to contain $1$, but since $d_{S}(x_{i-1}, 1)=1$, our assumption implies $d_{S}(x_{i-1},y)\leq d_{S}(x_{i},y), \text{ a contradiction}.$
\newline

\underline{Assume $1<y\leq \frac{3}{2}$}: 
There exists the edge between $1$ and $2$. Therefore, $g_{x_{i},y}$ needs to contain $2$ since $g_{x_{i},y}$ does not contain $1$. Furthermore, we notice that there exists the edge between $\frac{3}{2}$ and $1$, so $g_{x_{i},y}$ needs to contain $\frac{3}{2}$. However, since $d_{S}(x_{i-1},\frac{3}{2})=2$ as $x_{i-1}=0, 1, \frac{3}{2}$ is a geodesic, our assumption implies $d_{S}(x_{i-1},y)\leq d_{S}(x_{i},y),\text{ a contradiction}.$
\end{proof}

We show 

\begin{lemma}\label{ru}
Suppose $\x(S)=1$. Let $x,y\in C(S)$ and $g_{x,y}=\{x_{i}\}$ such that $d_{S}(x,x_{i})=i$ for all $i$. If $d_{R(x_{i})}(x,y)=L$, then, for all $0<i<d_{S}(x,y)$, we have $$L-2M \leq d_{R(x_{i})}(x_{i-1},\lfloor y \rfloor)\leq L+2M$$ $$\text{ or }$$ $$ L-2M \leq d_{R(x_{i})}(x_{i-1},\lceil y \rceil)\leq L+2M,$$ where $x_{i}=\frac{1}{0}$.
\end{lemma} 

\begin{proof}
To prove the statement, it suffices to show the following.
\begin{itemize}
\item $L-M \leq d_{R(x_{i})}(x_{i-1},y)\leq L+M.$
\item $d_{R(x_{i})}(y,\lfloor y \rfloor)\leq M \text{ or } d_{R(x_{i})}(y,\lceil y \rceil)\leq M.$
\end{itemize}


\underline{First inequality}:
Since $\pi_{R(x_{i})}(x_{j}) \neq \emptyset$ for all $j<i$, we have $d_{R(x_{i})}(x,x_{i-1})\leq M$ by Theorem \ref{BGIT}. Therefore, we have
\begin{itemize}
\item $d_{R(x_{i})}(x_{i-1},y) \geq d_{R(x_{i})}(x,y)- d_{R(x_{i})}(x,x_{i-1})\geq L-M .$
\item $d_{R(x_{i})}(x_{i-1},y) \leq d_{R(x_{i})}(x_{i-1},x)+ d_{R(x_{i})}(x,y) \leq M+L.$
\end{itemize}

\underline{Second inequality}: Recall that we assumed $x_{i}=\frac{1}{0}.$ There are two cases.

If $i=d_{S}(x,y)-1$, then since the set of all vertices which are distance $1$ apart from $x_{i}$ is $\mathbb{Z}$, $$y=\lfloor y \rfloor \text{ or } y=\lceil y \rceil. \Longrightarrow d_{R(x_{i})}(y,\lfloor y \rfloor)=0 \text{ or } d_{R(x_{i})}(y,\lceil y \rceil)=0.$$

If $i<d_{S}(x,y)-1$, then $y\neq \lfloor y \rfloor$ and $y\neq \lceil y \rceil$. Consider the two intervals obtained by removing $\{\lceil y \rceil ,\lfloor y\rfloor\}$ from $S^{1}$. Let $I_{y}$ be one of those intervals which contains $y$. See Fig. \ref{lu}. Since there exists the edge between $\lceil y \rceil$ and $\lfloor y\rfloor$, $g_{x_{i},y}$ needs to contain $\lfloor y \rfloor$ or $\lceil y \rceil$. This implies $$x_{i}\notin g_{\lfloor y \rfloor,y} \text{ or }x_{i}\notin g_{\lceil y \rceil ,y}. \stackrel{\text{Theorem \ref{BGIT}}}{ \Longrightarrow} d_{R(x_{i})}(y,\lfloor y \rfloor)\leq M \text{ or } d_{R(x_{i})}(y,\lceil y \rceil)\leq M.$$

\begin{figure}[h]
 \begin{center}
  \includegraphics[width=6cm,height =5cm]{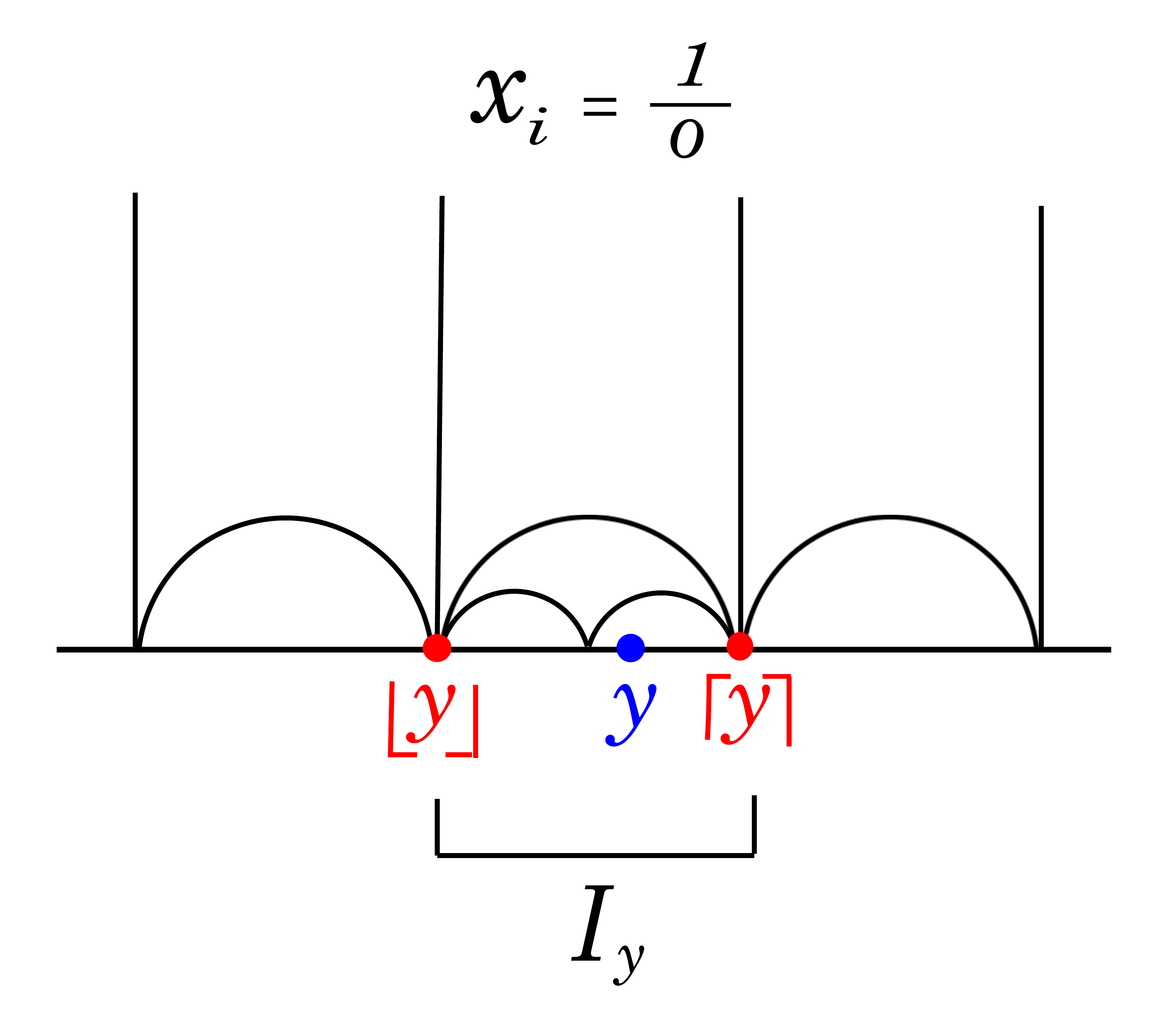}
   \end{center}
    \caption{When $i<d_{S}(x,y)-1$, we have $x_{i}\notin g_{\lfloor y \rfloor,y} \text{ or }x_{i}\notin g_{\lceil y \rceil ,y}$.}
   \label{lu}
\end{figure}

\end{proof}

We have the following key theorem.

\begin{theorem}\label{4}
Suppose $\x(S)=1$. Let $x,y\in C(S)$ and $g_{x,y}=\{x_{i}\}$ such that $d_{S}(x,x_{i})=i$ for all $i$. If $d_{R(x_{i})}(x,y)=L$, then, for all $0<i<d_{S}(x,y)$, we have $$L-2M-3 \leq\frac{i(x_{i-1}, y)}{i(x_{i}, y)}  \leq 2(L+2M).$$ 
\end{theorem} 
\begin{proof}
We may assume that $x_{i-1}=\frac{0}{1}$ and $x_{i}=\frac{1}{0}$. (If we do not have $x_{i-1}=\frac{0}{1}$ and $x_{i}=\frac{1}{0}$ initially, we can find a mapping class which acts on $g_{x,y}$ so that we have them.) Therefore, $|y|=\frac{i(x_{i-1}, y)}{i(x_{i}, y)}$. 
By Lemma \ref{2}, we observe that $\lfloor y\rfloor  \neq 0$ and $\lceil y \rceil \neq 0$. Let $n \in \{\lfloor y \rfloor,\lceil y \rceil \}.$ Since $n=T_{x_{i}} ^{n}(x_{i-1}),H_{x_{i}} ^{n}(x_{i-1})$ if $S=S_{1,1},S_{0,4}$ respectively, by Lemma \ref{min} we have 
$$d_{R(x_{i})}(x_{i-1},n)=
\left\{\begin{array}{ll}
		|n|+2  & \mbox{if } S=S_{1,1}. \\\\
		\Big \lfloor \frac{|n|}{2} \Big \rfloor+2  & \mbox{if }  S=S_{0,4}.
	\end{array}\right.$$
Since we always have $|n|-1\leq |y|\leq |n|+1$ (independent of $y>0$ or $y<0$), by the above observation, we have 
\begin{itemize}
\item $d_{R(x_{i})}(x_{i-1},n)-3 \leq |n|-1\leq |y| \leq |n|+1\leq d_{R(x_{i})}(x_{i-1},n) \text{ if } S=S_{1,1}.$
\item $2\cdot (d_{R(x_{i})}(x_{i-1},n)-3) \leq|n|-1\leq |y| \leq |n|+1\leq 2\cdot d_{R(x_{i})}(x_{i-1},n) \text{ if } S=S_{0,4}.$
\end{itemize}
Therefore, we always have $$d_{R(x_{i})}(x_{i-1},n)-3\leq |y|\leq 2\cdot d_{R(x_{i})}(x_{i-1},n).$$ Here we use Lemma \ref{ru}, and we have $$(L-2M)-3 \leq |y| \leq 2\cdot(L+2M).$$ 
We are done by letting $|y|=\frac{i(x_{i-1}, y)}{i(x_{i}, y)}.$

\end{proof}

\section{Theorem \ref{E''} }
By using Theorem \ref{4}, we prove Theorem \ref{E''}. We first observe the following lemma which enables us to pick any geodesic between $x,y\in C(S)$ to show Theorem \ref{E''}.

\begin{lemma}\label{mayw}
Suppose $\x(S)=1$ and $A\subsetneq S$, i.e., $A$ is an annulus. Let $x,y\in C(S)$. If $d_{A}(x,y)>M$ then $\p A$ is contained in every geodesic between $x$ and $y$. 
\end{lemma}
\begin{proof} 
By Theorem \ref{BGIT}, some vertex of $g_{x,y}$ does not project to $A$, which means that $\partial A $ is contained in $g_{x,y}$.
\end{proof} 

Now, we show 

\begin{theorem}\label{e}
Suppose $\x(S)=1$. Let $x,y\in C(S)$. We have 
\begin{itemize}
\item $\d \log i(x,y)\leq U^{+} \cdot \bigg( \sum_{Z\subseteq S}[ d_{Z}(x,y)]_{k}+\sum_{A\subseteq S} \log [d_{A}(x,y)]_{k} \bigg)+U^{+}$ for all $k\geq M$, where $U^{+}=k\cdot \log k^{2}\leq k^{3}$.

\item $\d \log i(x,y) \geq  \frac{1}{U^{-}}\cdot \bigg( \d\sum_{Z\subseteq S}[ d_{Z}(x,y)]_{k}+\sum_{A\subseteq S} \log [d_{A}(x,y)]_{k}  \bigg) -U^{-} $ for all $k\geq 3M,$ where $U^{-}=\max \{2,  (\log \frac{3}{2})^{-1}, \log \frac{3}{2} \}=2 .$
\end{itemize}
\end{theorem}

\begin{proof} 
Let $g_{x,y}=\{ x_{i}\}$ such that $d_{S}(x,x_{i})=i$ for all $i$. By Lemma \ref{mayw}, if $[d_{A}(x,y)]_{k}>0$ then $\p A=x_{j}$ for some $j$.

Let $d_{R(x_{i})}(x,y)=L_{i}$ for all $0<i<d_{S}(x,y)$. By Lemma \ref{2} and Theorem \ref{4}, we have $$\max \bigg\{ \frac{3}{2},L_{i}-2M-3\bigg\} \leq \frac{i(x_{i-1}, y)}{i(x_{i}, y)} \leq  2(L_{i}+2M).$$ 

Recall $M=200$.

\underline{\textbf{First statement}}: First, we define $L_{i}^{+} =k^{2}$ if $L_{i} \leq k$ and $L_{i}^{+} =k^{2} L_{i}$ if $L_{i}> k$. Then, since $k\geq M$, we have $2(L_{i}+2M)\leq L_{i}^{+}$; in particular, we have $$\frac{i(x_{i-1}, y)}{i(x_{i}, y)} \leq  2(L_{i}+2M) \leq L_{i}^{+}.$$
Therefore, we have $$\d \prod_{i=1}^{d_{S}(x,y)-1} \frac{i(x_{i-1}, y)}{i(x_{i}, y)} =\frac{i(x, y)}{i(x_{d_{S}(x,y)-1}, y)}  \leq \prod_{i=1}^{d_{S}(x,y)-1} L_{i}^{+}.$$
By taking $\log$, we have $$\d \log i(x,y) \leq \log i(x_{d_{S}(x,y)-1}, y)+\sum_{i=1}^{d_{S}(x,y)-1} \log L_{i}^{+}.$$
Since $i(x_{d_{S}(x,y)-1}, y)=1,2$ if $S=S_{1,1}, S_{0,4}$ respectively, we have 
\begin{eqnarray*}
\displaystyle \log i(x,y) &\leq& \log 2+ \sum_{i=1}^{d_{S}(x,y)-1} \log L_{i}^{+}\\\\&=& \log 2+ \log k^{2}\cdot (d_{S}(x,y)-1)  +  \bigg( \sum_{A\subseteq S }\log [d_{A}(x,y)]_{k}\bigg)\\\\&\leq& \log k^{2} \cdot d_{S}(x,y)+\bigg( \sum_{A \subseteq S}\log [d_{A}(x,y)]_{k}\bigg).
\end{eqnarray*}

If $[d_{S}(x,y)]_{k}>0$, then we have $$\d \log i(x,y) \leq \log k^{2} \cdot \bigg(  [d_{S}(x,y)]_{k} +\sum_{A \subseteq S}\log [d_{A}(x,y)]_{k}\bigg).$$

If $[d_{S}(x,y)]_{k}=0$, then we have $$\d \log i(x,y) \leq \bigg( [d_{S}(x,y)]_{k}+\sum_{A \subseteq S}\log [d_{A}(x,y)]_{k}\bigg)+k\cdot \log k^{2} .$$

\underline{\textbf{Second statement}}: The proof is analogous to the proof of the previous case; we briefly go over the proof. 

First, we define $L_{i}^{-}=\frac{3}{2}$ if $L_{i} \leq k$ and $L_{i}^{-}=\frac{3}{2}\sqrt{L_{i}}$ if $L_{i} > k$. Then, since $k\geq 3M$, we have $L_{i}^{-}\leq \max\{\frac{3}{2},L_{i}-2M-3\}$. Therefore, we have 
\begin{eqnarray*}
\displaystyle L_{i}^{-} \leq \frac{i(x_{i-1}, y)}{i(x_{i}, y)}. &\Longrightarrow& \sum_{i=1}^{d_{S}(x,y)-1} \log L_{i}^{-} +\log i(x_{d_{S}(x,y)-1}, y) \leq \log i(x,y).
\end{eqnarray*}
Lastly, let $i(x_{d_{S}(x,y)-1}, y)=1$, and observe
\begin{eqnarray*}
\displaystyle \sum_{i=1}^{d_{S}(x,y)-1} \log L_{i}^{-}= \log \frac{3}{2}\cdot (d_{S}(x,y)-1) + \frac{1}{2}\cdot \bigg( \sum_{A\subseteq S }\log [d_{A}(x,y)]_{k}\bigg).
\end{eqnarray*}
\end{proof}

\subsection{Intersections of arcs and curves}
The goal of this section is to observe a technical fact, Lemma \ref{i}. We use it in our inductive argument to prove Theorem \ref{E} in the next section.

Suppose $C$ is a subset of $S$; we let $S-C$ denote the set of complementary components of $C$ in $S$, which we treat as embedded subsurfaces in $S$.

We say that two curves $a$ and $b$ in $S$ form a \emph{bigon} if there is an embedded disk in $S$ whose boundary is the union of an interval of $a$ and an interval of $b$ meeting in exactly two points. 
It is well--known that two curves are in minimal position if and only if they do not form a bigon. This fact is called the \emph{bigon criteria} in \cite{FM}. It is also well--known that two curves $a$ and $b$ are in minimal position if and only if every arc obtained by $b \cap (S-a)$ is essential in $S-a$. 

We show
\begin{lemma}\label{i}
Let $x \in C(S)$ and $y\in A(S)$ such that they are in minimal position. 
\begin{itemize}
\item Suppose $\partial y$ is contained in two distinct boundary components of $S$. Then we have $i(x, \pi_{S}(y))=2\cdot i(x,y).$
\item Suppose $\partial y$ is contained in a single boundary component of $S$. Then we have $i(x, Y)= i(x,y)$ for some $Y\in \{ \pi_{S}(y)\}.$ 
\end{itemize}
In particular, we always have $i(x,y)\leq i(x, \pi_{S}(y)).$
\end{lemma}

\begin{proof}
It suffices to show that the arcs obtained by $\pi_{S}(y) \cap (S-x)$ are all essential in $S-x$. (If $\partial y$ lie on a boundary of $S$ then $\pi_{S}(y)$ can contain two components; in this case we need to show the essentiality of arcs for one of the components.)

Let $\{a_{i}\}$ be the set of arcs obtained by $y \cap (S-x)$; they are all essential in $S-x$ since $x$ and $y$ are in minimal position.
\newline

\underline{\textbf{First statement}}: Pick $a_{p},a_{q}\in \{a_{i}\} $ such that one of $\partial a_{p}$ lies on $B\in \partial S$ and one of $\partial a_{q}$ lies on $B'\in \partial S$. Let $R_{p}=\partial  R( a_{p} \cup B)$ and $R_{q}=\partial  R( a_{q}\cup B')$; we note $R_{p}, R_{q}\in \{\pi_{S}(y) \cap (S-x)\}.$
By the definition of subsurface projections, every element in $\{\pi_{S}(y) \cap (S-x)\} \setminus \{R_{p},R_{q}\}$ is parallel to some element in $\{a_{i}\} \setminus \{a_{p},a_{q}\}$, which is originally essential in $S-x$.

We observe that if $R_{p}$ and $R_{q}$ are both essential in $S-x$, then every other element in $\{\pi_{S}(y) \cap (S-x)\}$ stays essential in $S-x$.
Indeed, $R_{p}$ is an essential arc in $S-x$, otherwise $B$ and $x$ would be isotopic. See Fig. \ref{8}. The same argument works for $R_{q}.$
\newline

\underline{\textbf{Second statement}}: Pick $a_{p},a_{q} \in \{a_{i}\}$ such that $c\in \partial a_{p}$ and $c'\in \partial a_{q}$ lie on $B\in \partial S$. Now, we have two cases, i.e.,  their other boundaries lie on two distinct boundaries (Case 1) or the same boundary (Case 2) of $S-x$, which come from cutting $S$ along $x$. 
Let $B_{1}$ and $B_{2}$ be the closure of the intervals of $B$, which are obtained by removing $\{c, c'\}$ from $B$. See Fig. \ref{9}.

For both cases, it suffices to check the essentiality of $\partial  R( a_{p}\cup B_{1} \cup a_{q})=R_{1} $ or $\partial  R( a_{p}\cup B_{2} \cup a_{q}) =R_{2}$ in $S-x$.

\underline{Case 1}: $R_{1}$ and $R_{2}$ are both essential in $S-x$ since the endpoints of them are contained in two distinct boundaries of $S-x$.

\underline{Case 2}: $R_{1}$ or $R_{2}$ needs to be essential in $S-x$, otherwise $B$ and $x$ would be isotopic. 

\begin{figure}[htbp]
 \begin{center}
  \includegraphics[width=2.8cm,height =3cm]{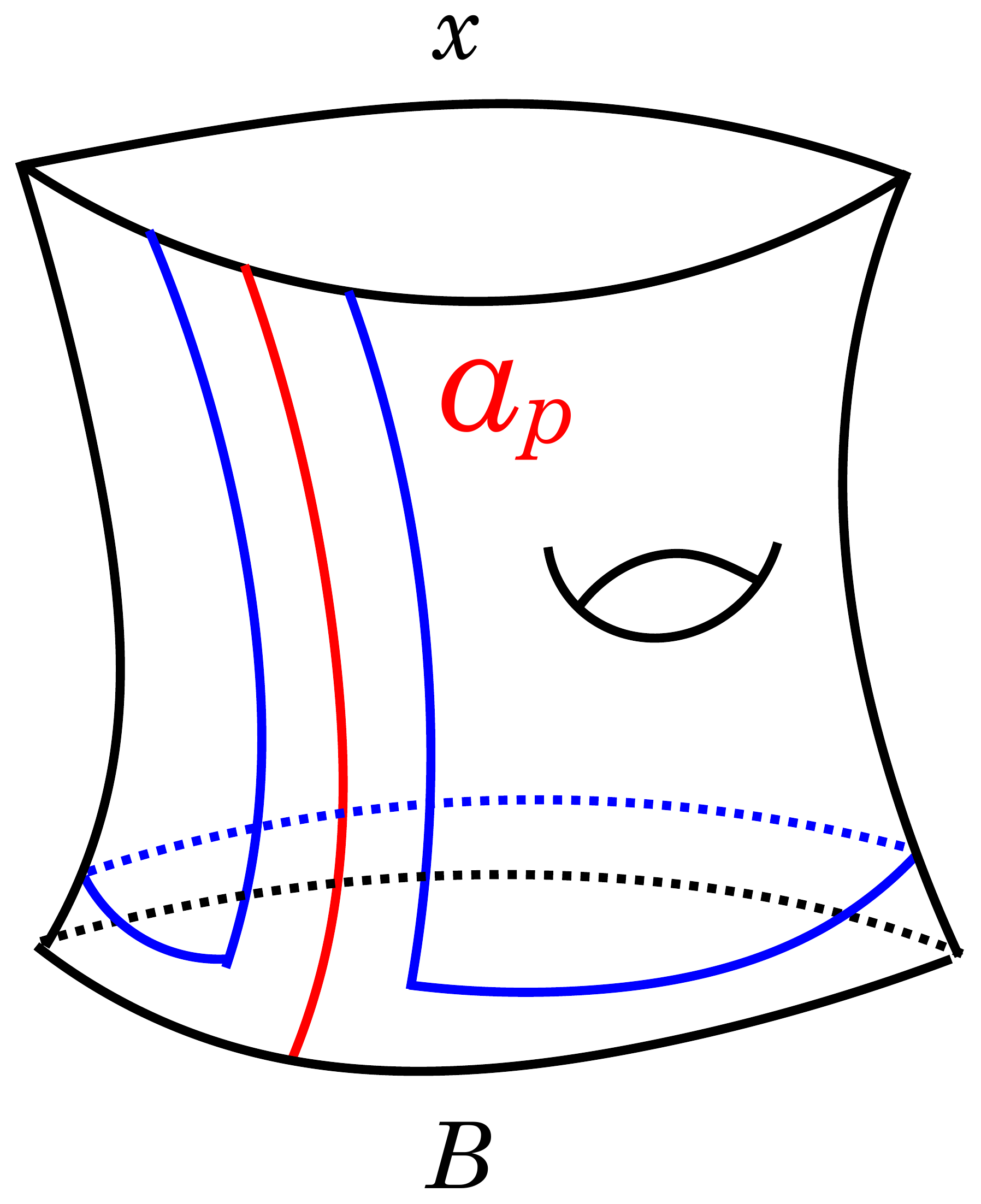}
   \end{center}
   \caption{$R_{p}=\partial R( a_{p}\cup B) \in \{\pi_{S}(y) \cap (S-x)\}$ in $S-x$.}
  \label{8}

\end{figure}

\begin{figure}[htbp]
 \begin{center}
  \includegraphics[width=8cm,height =3.5cm]{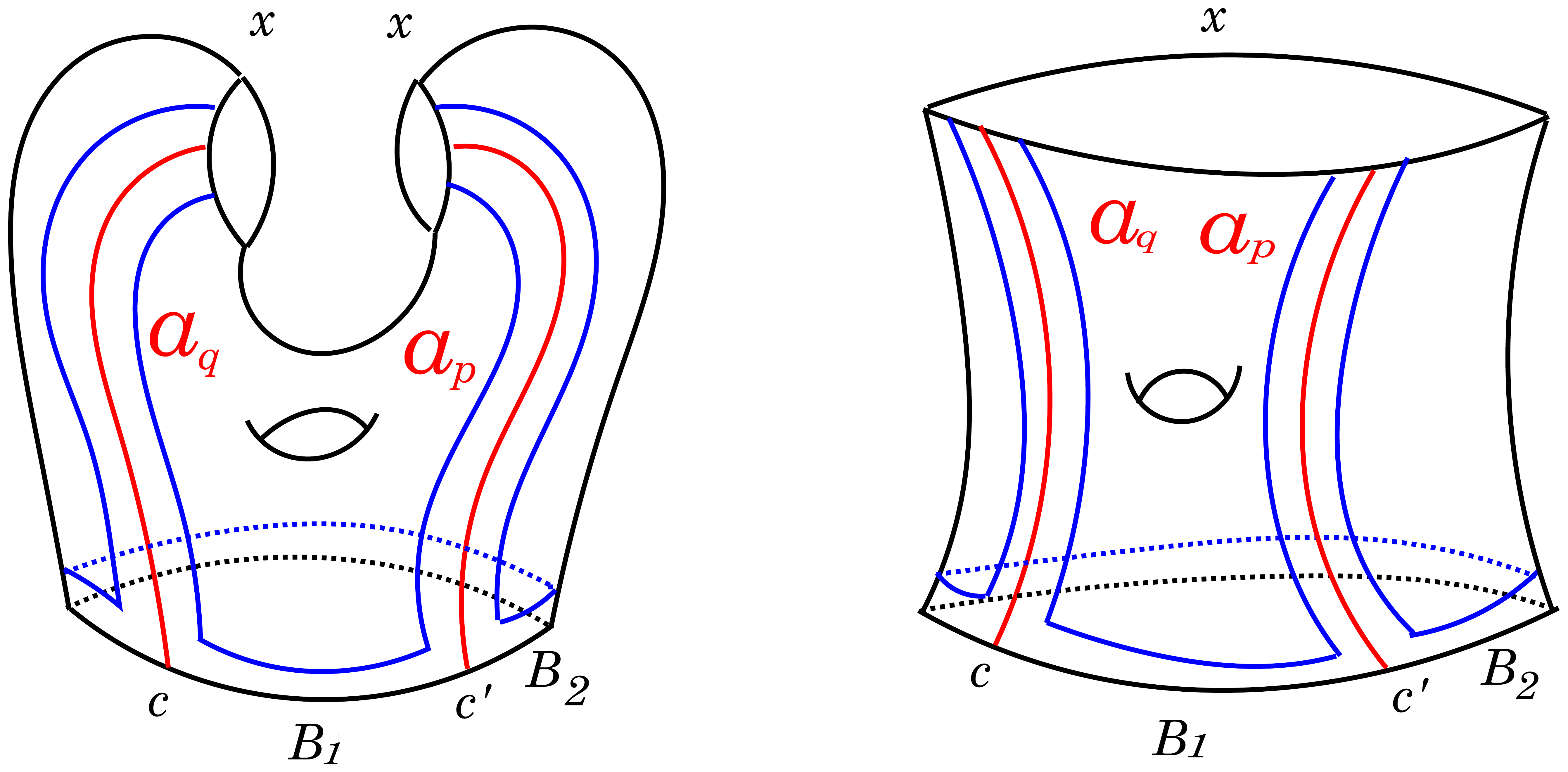}
   \end{center}
    \caption{Case 1 is on left and Case 2 is on right. $B=B_{1}\cup B_{2}$.}
 \label{9}
\end{figure}

\end{proof}

\section{Theorem \ref{E}}

The goal of this section is completing the proof of Theorem \ref{E}. It will be proved by an inductive argument on the complexity; the base case was proved in Theorem \ref{E''}. We note Lemma \ref{tama} is the key for this induction. 

In order to simplify our notations for the rest of this section, we define the following. 

\begin{definition}\label{definition}
Suppose $\x(S)>1$.
\begin{itemize}
\item Let $x,y\in C(S)$ and $g_{x,y} = \{x_{i}\}$ such that $d_{S}(x,x_{i}) = i$ for all $i$. We define 

\begin{itemize}
\item $\d G_{i}(k)=\sum_{Z\subseteq S-x_{i}} [ d_{Z}(x_{i-1},y)]_{k}+\sum_{A\subseteq S-x_{i}} \log [d_{A}(x_{i-1},y)]_{k}.$

\item $\d \G(k)=\sum_{i=1}^{d_{S}(x,y)-1} G_{i}(k).$

\end{itemize}

\item Let $N_{S_{g,n}}$ be the smallest cut-off constant so that Theorem \ref{E} holds for $S_{g,n}$; take $k\geq N_{S_{g,n}} $, we let $P_{S_{g,n}}(k)$ and $Q_{S_{g,n}}(k)$ be the multiplicative and additive constants on Theorem \ref{E} respectively, i.e., $$\log i(x,y) \leq P_{S_{g,n}}(k) \cdot \bigg( \sum_{Z\subseteq S_{g,n}}[ d_{Z}(x,y)]_{k}+\sum_{A\subseteq S_{g,n}} \log [d_{A}(x,y)]_{k}\bigg) + Q_{S_{g,n}}(k).$$ We let $$N=\max \{ N_{S_{g,n}} | \x(S_{g,n}) < \x(S)\}.$$

For $k \geq N$, we let
\begin{itemize}
\item $P(k)=\max \{ P_{S_{g,n}}(k) | \x(S_{g,n}) < \x(S)\}. $
\item $Q(k)=\max \{ Q_{S_{g,n}}(k) | \x(S_{g,n}) < \x(S)\}. $
\end{itemize}


In particular, $P(k),Q(k) \geq  k^{3}$ and $N\geq M$. See Theorem \ref{e}.
\end{itemize}
\end{definition}

As a warm--up of the proof of Lemma \ref{tama}, we first show

\begin{lemma}\label{tamas}
Suppose $\x(S)>1$. Let $x,y\in C(S)$ such that $d_{S}(x,y)= 2$. For all $n\geq N$, we have $$\log i(x,y)\stackrel{P(n), Q(n)}{\leq} G_{1}(n).$$
\end{lemma}
\begin{proof}
Let $g_{x,y} = \{x_{i}\}$ such that $d_{S}(x,x_{i}) = i$ for all $i$. We let $ S-x_{1}=\{S',S''\}$. (If $x_{1}$ does not separate $S$, then we treat $ S-x_{1}=\{S'\}$.) Since $i(x,y)>0$, $x$ and $y$ are contained in the same component, say $S'$; we can use the inductive hypothesis, and we have $$\log i(x,y) \stackrel{P(n), Q(n)}{\leq} \sum_{Z\subseteq S'} [ d_{Z}(x,y)]_{n}+\sum_{A\subseteq S'} \log [d_{A}(x,y)]_{n}.$$

Recall that the subsurfaces which are considered in $G_{1}(n)$ are taken from $S-x_{1}$. However, if $W\subseteq S''$, then neither $x$ nor $y$ projects to $W$; therefore, the right--hand side of the above is same as $G_{1}(n).$
\end{proof}
We state the following algebraic identity. Recall that we treat $\log 0=0.$

\begin{lemma}\label{algebra}
Suppose $\{m_{i}\}_{i=1}^{l}\in \mathbb{N}_{\geq 0}$. Then, $\d \log \bigg(\sum_{i=1}^{l} m_{i}\bigg) \leq   \bigg(\sum_{i=1}^{l} \log m_{i} \bigg)+\log l.$
\end{lemma}
\begin{proof}
If $m_{i}=0$ for all $i$, then we are done. If not, without loss of generality, we assume $m_{i}>0$ for all i; we observe $\d \sum_{i=1}^{l}m_{i}\leq \big(\max_{i} m_{i}\big) \cdot l \leq \bigg( \prod_{i=1}^{l}m_{i} \bigg) \cdot l.$
\end{proof}

Now, we have the following key fact. 

\begin{lemma}\label{tama}
Suppose $\x(S)>1$. Let $x,y\in C(S)$. For all $n\geq N$, we have $$\log i(x,y)\stackrel{K(n), C(n)\cdot d_{S}(x,y)}{\leq} \G(n)$$ where $K(n)=6|\chi(S)| \cdot P(n)$ and $C(n)=6|\chi(S)| \cdot (Q(n)+1)$.
\end{lemma}

\begin{proof}
Let $g_{x,y} = \{x_{i}\}$ such that $d_{S}(x,x_{i}) = i$ for all $i$. Let $S'$ be a complementary component of $x_{i+1}$ such that $x_{i}\in C(S')$. Recall Lemma \ref{kp}, we have $$|i_{S'}( y)=\{a_{j}\}|\leq 3|\chi(S)| \text{ and } |\pi_{S'}( y)=\{c_{j}\}|\leq 6|\chi(S)|.$$ 
Let $n_{j}$ be the number of arcs obtained by $y\cap S'$ which are isotopic to $a_{j}$.

We first state the main steps of the proof.
\begin{itemize}
\item \textbf{Step 1 (Topological step):}
By making topological observations, we show 
$$\log i(x_{i},y)-\log i(x_{i+1},y)\prec \sum_{j} \log i(x_{i},c_{j}) \text{ for all }i<d_{S}(x,y)-2.$$
\item \textbf{Step 2 (Inductive step by complexity):}
By using inductive hypothesis on the complexity, we show $$ \sum_{j} \log i(x_{i},c_{j})\prec G_{i+1}(n) \text{ for all }i<d_{S}(x,y)-2.$$
\item \textbf{Step 3 (Deriving $K$ and $C$):}
By the same proof of Lemma \ref{tamas}, we have $$\d \log i(x_{d_{S}(x,y)-2},y)\prec G_{d_{S}(x,y)-1}(n).$$ With this observation and the results from Step 1 and Step 2, we derive $K$ and $C$.
\end{itemize}

Now, we start the proof. For Step 1 and Step 2, we assume $i=0$ for simplicity; the same proof works for all $i<d_{S}(x,y)-2.$ Also, by abuse of notation, we denote $P(n),Q(n)$ by $P,Q$ respectively for the rest of the proof.

\underline{\textbf{Step 1:}}
We first observe 
\begin{itemize}
\item $\displaystyle  i(x,y)=\sum_{j} n_{j} \cdot i(x,a_{j})$ by the definitions of $a_{j}$ and $n_{j}$. 
\item $\displaystyle i(x_{1},y)=\sum_{j} n_{j}$ when $x_{1}$ does not separate $S$ and $\displaystyle i(x_{1},y) = 2\cdot  \sum_{j} n_{j}$ when $x_{1}$ separates $S$. 
\end{itemize} 
Therefore, we have 
\begin{eqnarray*}
&&i(x,y)=\sum_{j} n_{j} \cdot i(x,a_{j}). \\\\& \stackrel{n_{j}\leq i(x_{1},y)} {\Longrightarrow}& i(x,y)\leq i(x_{1},y) \cdot \Big(\sum_{j}   i(x,a_{j})\Big).\\\\&\stackrel{\text {Lemma \ref{i}}} {\Longrightarrow}&i(x,y)\leq i(x_{1},y) \cdot \Big( \sum_{j} i(x,c_{j}) \Big).\\\\&\Longrightarrow&\log i(x,y)-\log i(x_{1},y) \leq   \log \Big( \sum_{j} i(x,c_{j}) \Big).
\end{eqnarray*}
With Lemma \ref{algebra}, we have
\begin{equation}\label{(2)}
 \log i(x,y)-\log i(x_{1},y) \leq   \log \Big( \sum_{j} i(x,c_{j}) \Big) \leq \sum_{j} \log  i(x,c_{j})+6|\chi(S)|.
\end{equation}

\underline{\textbf{Step 2:}}
Since $x,c_{j}$ are contained in the same complementary component of $x_{1}$, as in the proof of Lemma \ref{tamas}, we can use the inductive hypothesis on the complexity for all $n\geq N$. We have
$$\log i(x,c_{j}) \stackrel{P,  Q }{\leq} \sum_{Z\subseteq S-x_{1}}[ d_{Z}(x,c_{j})]_{n}+\sum_{A\subseteq S-x_{1}} \log [d_{A}(x, c_{j})]_{n}\text{ for all }j.$$
Since the right--hand side of the above is bounded by $G_{1}(n)$ and $|\{c_{j}\}|\leq  6|\chi(S)|$, we have $$ \sum_{j} \log i(x,c_{j}) \stackrel{6|\chi(S)| \cdot P, 6|\chi(S)| \cdot Q  }{\leq} G_{1}(n).$$
Take $P' = 6|\chi(S)| \cdot P$ and $Q'=6|\chi(S)| \cdot (Q+1)$; with (5.1), we have  

\begin{equation}\label{(2)}
 \log i(x,y)- \log i(x_{1},y) \stackrel{P', Q' }{\leq} G_{1}(n).
\end{equation}

\underline{ \textbf{Step 3:}}
The same arguments on the previous steps yield the desired inequality, that is (5.2), $\text{for all }i<d_{S}(x,y)-2$. Namely, we have   
$$\log i(x_{i},y)-\log i(x_{i+1},y) \stackrel{P', Q' }{\leq}G_{i+1}(n)\text{ for all }i<d_{S}(x,y)-2.$$
Therefore, we have 
$$ \sum_{i=0}^{d_{S}(x,y)-3} \log i(x_{i},y)-\log i(x_{i+1},y)\stackrel{P', Q'\cdot (d_{S}(x,y)-2) }{\leq}\sum_{i=0}^{d_{S}(x,y)-3}G_{i+1}(n),$$
which is that
\begin{eqnarray}
\log i(x,y)-\log i(x_{d_{S}(x,y)-2},y) \stackrel{P', Q'\cdot (d_{S}(x,y)-2) }{\leq} \G(n)-G_{d_{S}(x,y)-1}(n).
\end{eqnarray}

Apply the same proof of Lemma \ref{tamas} on $\log i(x_{d_{S}(x,y)-2},y)$, and obtain 
\begin{equation}
\d \log i(x_{d_{S}(x,y)-2},y)\stackrel{P, Q}{\leq} G_{d_{S}(x,y)-1}(n).
\end{equation}
Since $P'\geq P$ and $Q'\geq Q$, by (5.3) and (5.4) we have $$\displaystyle  \log i(x,y) \stackrel{P', Q'\cdot d_{S}(x,y)}\leq \G(n).$$ 

Lastly, we let $K=P'=6|\chi(S)| \cdot P$ and $C=Q'=6|\chi(S)| \cdot (Q+1)$, and we are done.
\end{proof}

We obtain the following important corollary; once we have it, we compute the constants, and that is Theorem \ref{E}. 
\begin{corollary}\label{last}
Suppose $\x(S)>1$. Let $x,y\in C(S)$. For all $k> N-M$, we have $$\log i(x,y) \leq A(k) \cdot  \bigg(   \sum_{Z\subseteq S}[ d_{Z}(x,y)]_{k}+\sum_{A\subseteq S} \log [d_{A}(x,y)]_{k} \bigg)+ B(k)$$ where $A(k)=\max\{6 M\cdot K(k+M),C(k+M) \}$ and $B(k)=k\cdot C(k+M).$

\end{corollary}
\begin{proof}
Let $g_{x,y} = \{x_{i}\}$ such that $d_{S}(x,x_{i}) = i$ for all $i$. We define
\begin{itemize}
\item $\displaystyle H_{i}(n)=\sum_{Z\subseteq S-x_{i}} [ d_{Z}(x,y)]_{n}+\sum_{A\subseteq S-x_{i}} \log [d_{A}(x,y)]_{n}.$
\item $\d \H(n)=\sum_{i=1}^{d_{S}(x,y)-1} H_{i}(n).$
\end{itemize}

For the rest of the proof, we assume $g_{x,y}$ is a tight geodesic.

Recall Definition \ref{definition}, we first show $\displaystyle  G_{i}(k+M) \leq 2M\cdot H_{i}(k) \text{ for all } i.$ Suppose $W\subseteq S-x_{i}$ such that $[d_{W}(x_{i-1},y)]_{k+M}>0$; then we have $d_{W}(x,x_{i-1}) \leq M$ by tightness, i.e., Lemma \ref{sss}. Therefore, we have $[d_{W}(x,y)]_{k}>0$; in particular, we have
 
\begin{itemize}
\item $[d_{W}(x_{i-1},y)]_{k+M} \leq 2M\cdot [d_{W}(x,y)]_{k}.$
\item $\log [d_{W}(x_{i-1},y)]_{k+M}  \leq 2M\cdot \log [d_{W}(x,y)]_{k}.$
\end{itemize}
Thus, we obtain
$$G_{i}(k+M) \leq 2M\cdot H_{i}(k)  \text{ for all }  i. \Longrightarrow \G(k+M) \leq 2M\cdot  \H(k).$$

Suppose $W \subseteq S-x_{i}$; then $W$ can be contained in the compliment of at most three vertices (including $x_{i}$) in $g_{x,y}$ by Lemma \ref{distance3}. Therefore, we have 
$$\d \H(k) \leq 3\cdot \bigg(\sum_{Z\subsetneq S}[ d_{Z}(x,y)]_{k}+\sum_{A\subsetneq S} \log [d_{A}(x,y)]_{k} \bigg).$$ 

Lastly, since $k+M\geq N$, we can apply Lemma \ref{tama} on $\log i(x,y)$; by abuse of notation, we denote $K(k+M),C(k+M)$ by $K,C$ respectively for the rest of the proof. We have $$\log i(x,y)\leq K\cdot \G(k+M)+ C\cdot d_{S}(x,y).$$ 

All together, we have 
\begin{eqnarray*}
\log i(x,y) &\leq& K\cdot \G(k+M)+ C\cdot d_{S}(x,y) \\\\&\leq& 2M\cdot K\cdot  \H(k)+ C\cdot d_{S}(x,y) \\\\&\leq& 6M\cdot K\cdot \bigg(\sum_{Z\subsetneq S}[ d_{Z}(x,y)]_{k}+\sum_{A\subsetneq S} \log [d_{A}(x,y)]_{k} \bigg) +  C\cdot d_{S}(x,y). 
\end{eqnarray*}
If $[d_{S}(x,y)]_{k}>0$, then we have $$\d \log i(x,y) \leq \max\{6M\cdot K,C\} \cdot  \bigg(   \sum_{Z\subseteq S}[ d_{Z}(x,y)]_{k}+\sum_{A\subseteq S} \log [d_{A}(x,y)]_{k} \bigg).$$
If $[d_{S}(x,y)]_{k}=0$, then we have $$\d \log i(x,y) \leq 6M\cdot K \cdot  \bigg(   \sum_{Z\subseteq S}[ d_{Z}(x,y)]_{k}+\sum_{A\subseteq S} \log [d_{A}(x,y)]_{k} \bigg)+ k\cdot C.$$

\end{proof}

We observe from Corollary \ref{last} that we can take any positive integer as the minimum cut--off constant for $\x(S)>1$. Now, we complete the proof of Theorem \ref{E}.

\begin{theorem}[Effective version of Corollary \ref{last}]\label{CE}
Suppose $\x(S)>1$. Let $x,y\in C(S)$. For all $k> 0$, we have $$\log i(x,y)\leq V(k) \cdot \bigg( \sum_{Z\subseteq S}[ d_{Z}(x,y)]_{k}+\sum_{A\subseteq S} \log [d_{A}(x,y)]_{k}\bigg) + V(k)$$ where 
$V(k)=\big( M^{2}|\chi(S)| \cdot (k+\x(S)\cdot M)  \big)^{\x(S)+2} .$
\end{theorem}
\begin{proof}
Recall Corollary \ref{last}; $A(k)=\max\{6M\cdot K(k+M),C(k+M) \}$ and $B(k)=k\cdot C(k+M).$
Combining with Lemma \ref{tama}, we have 
\begin{itemize}
\item$6M\cdot K(k+M)=36M |\chi(S)| \cdot P(k+M) \leq M^{2}|\chi(S)| \cdot P(k+M).$ 
\item $C(k+M)=6|\chi(S)| \cdot (Q(k+M)+1)\leq  M|\chi(S)| \cdot Q(k+M)$.  
\end{itemize}
Therefore, it suffices to let 
\begin{itemize}
\item $A(k)=\max \{  M^{2}|\chi(S)| \cdot P(k+M), M^{2}|\chi(S)| \cdot Q(k+M)\} $
\item  $B(k)=k \cdot M^{2}|\chi(S)| \cdot Q(k+M)$ 
\end{itemize}
in Corollary \ref{last}.

Recall Theorem \ref{E''}; the multiplicative and additive constants can be taken to be $k^{3}$ when $\x(S)=1$. Hence, it suffices to understand the growth of $B(k)$. One can check $$B(k)\leq \big( M^{2}|\chi(S)| \cdot (k+\x(S)\cdot M)  \big)^{\x(S)+2}. $$
\end{proof}

\section{Application}
We show that given $x,y\in C(S)$, if $\{x_{i}\}$ is a tight multigeodesic between $x$ and $y$ such that $d_{S}(x,x_{i})=i$ for all $i$, then $i(x_{i+1},y)$ decreases from $i(x_{i},y)$ with a multiplicative factor for all $i$ where the multiplicative constant depends only on the surface. Indeed, this type of result has been obtained before by Shackleton; we recall Shackleton's result from \cite{SHA}. Let $F:\mathbb{N} \rightarrow \mathbb{N}$, $$F(n)=n \cdot T^{\lfloor 2 \log n \rfloor}$$ where $T=4^{5} \cdot \x(S)^{3}$. Let $F^{j}=\underbrace{F\circ F\circ \cdots \circ F}_{j \text{ many } F \text{'s} }.$

Shackleton showed 

\begin{theorem}[\cite{SHA}]\label{tabi}
Suppose $\x(S)>1$. Let $x,y\in C(S)$ and $g_{x,y} = \{x_{i}\}$ be a tight multigeodesic such that $d_{S}(x,x_{i}) = i$ for all $i$, then 
$$ i(x_{i},y) \leq F^{i}(i(x,y))  \text{ for all }i.$$
\end{theorem}

We improve on Shackleton's result; we show that $F$ in the above can be taken to be a linear function $G:\mathbb{N} \rightarrow \mathbb{N}$, $$G(n)=R\cdot n$$ where $R=\x(S)\cdot  2^{V(M)}$.

We show
\begin{theorem}
Suppose $\x(S)\geq1$. Let $x,y\in C(S)$ and $g_{x,y} = \{x_{i}\}$ be a tight multigeodesic such that $d_{S}(x,x_{i}) = i$ for all $i$, then
$$ i(x_{i},y) \leq R^{i}\cdot i(x,y)  \text{ for all }i.$$
\end{theorem}
\begin{proof}
We prove $$i(x_{i+1},y) \leq R\cdot i(x_{i},y) \text{ for all }i,$$ which gives the statement of this theorem.

Suppose $\x(S)=1.$ We can take $R=\frac{2}{3}$ by Lemma \ref{2}.

Suppose $\x(S)>1.$ Let $b\in x_{i+1}$ and let $S'$ be a complementary component of $x_{i}$ such that $b\in C(S')$. We note $\x(S')\geq 1.$ Lastly, we let $\{a_{i}\}$ be the set of arcs obtained by $y\cap S'.$ 

Let $a_{t}\in \{a_{i}\}$; by Lemma \ref{i}, we can choose a component $A_{t} \in \{\pi_{S'}(a_{t})\}$ such that $i(b,a_{t}) \leq i(b,A_{t}).$
With Theorem \ref{E}, we have 
\begin{eqnarray*}
\log i(b,a_{t})&\leq& \log i(b,A_{t})\\&\leq & V(k) \cdot  \bigg( \sum_{Z\subseteq S'}[ d_{Z}(b,A_{t})]_{k}+\sum_{A\subseteq S'} \log [d_{A}(b,A_{t})]_{k}\bigg)  + V(k).
\end{eqnarray*}
Now, we show that $ [d_{W}(b, A_{t})]_{M}=0$ for all $W\subseteq S'$. Since $\pi_{W}(x_{i})=\emptyset$, we have $\pi_{W}(x_{i+2})\neq \emptyset$ by tightness. With Lemma \ref{distance3} and Theorem \ref{BGIT}, we have $[d_{W}(b, y)]_{M}=0.$ By the definition of subsurface projections, we have $d_{W}(b, A_{t})\leq d_{W}(b,y)$; therefore, $ [d_{W}(b, A_{t})]_{M}=0$.
All together, we have $$\log i(b,a_{t})\leq \log i(b,A_{t}) \leq V(M) \cdot  0  + V(M);$$ and we obtain

\begin{equation}
i(b,a_{t})\leq 2^{V(M)}. \tag{$\dagger$}
\end{equation}

Since $\displaystyle i(x_{i+1},y)= \sum_{b\in x_{i+1}} i(b,y)$ and $\displaystyle i(b,y)=\sum_{t} i(b,a_{t}),$ we have 
$$i(x_{i+1},y)\leq  \sum_{b\in x_{i+1}} \bigg(\sum_{t} i(b,a_{t}) \bigg) \leq \sum_{b\in x_{i+1}} \bigg( i(b,a_{t}) \cdot |\{a_{t}\}|  \bigg).$$
We notice that each $a_{t}$ contributes to $i(x_{i},y)$; we have $|\{a_{t}\}|\leq i(x_{i},y).$ Therefore, we have 
$$i(x_{i+1},y)\leq  \sum_{b\in x_{i+1}}  \bigg( i(b,a_{t}) \cdot i(x_{i},y) \bigg).$$
Since $x_{i+1}$ is a multicurve, it can contain at most $\x(S)$ many curves, so we have
$$i(x_{i+1},y) \leq \x(S)\cdot  \bigg( i(b,a_{t})\cdot i(x_{i},y) \bigg).$$ With $(\dagger)$, we obtain $$i(x_{i+1},y) \leq \x(S)\cdot 2^{V(M)}\cdot i(x_{i},y).$$
Lastly, we let $$R=\x(S)\cdot 2^{V(M)}.$$

\end{proof}

\bibliographystyle{plain}
\bibliography{references.bib}

\end{document}